\def\Bbb{\mathbb}

\def\dist{{\rm{dist}}}

\def\dist{{\rm{dist}}}

\def\Bbb{\mathbb}
\def\reals{\Bbb R}
\def\complex{\Bbb C}
\def\disk{\Bbb D}

\def\integers{\Bbb Z}

\def\rhp{{\Bbb H}_r}

\def\classS{\cal S}
\def\classB{\cal B}

\def\disk{\Bbb D}
\def\cal{\mathcal}

\documentclass[12pt]{amsart}  
\usepackage{amssymb}    

\usepackage{graphicx,amssymb,amsthm,verbatim,amsfonts}

\theoremstyle{plain}                    
\newtheorem{thm}{Theorem}[section]

\newtheorem{lemma}[thm]{Lemma}

\newcounter{ques}
   
\numberwithin{equation}{section}

\begin{document}
\baselineskip=18pt


%

\title [  The order conjecture fails in $\classS$ ]
          { The order conjecture fails in $\classS$ }

\author {Christopher J. Bishop}
\address{C.J. Bishop\\
         Mathematics Department\\
         SUNY at Stony Brook \\
         Stony Brook, NY 11794-3651}
\email {bishop@math.sunysb.edu}
\thanks{The  author is partially supported by NSF Grant DMS 13-05233.
        }

\date{June 26, 2013}
\maketitle


\begin{abstract}
We construct an entire function  $f$  with only  three
singular  values whose order can change under a 
quasiconformal equivalence.
\end{abstract}

\clearpage


\setcounter{page}{1}
\renewcommand{\thepage}{\arabic{page}}
\section{Introduction} \label{Intro}

If $f$ is entire, let $S(f)$ denote the singular set of $f$, that is, 
  the closure of its  critical values and finite asymptotic values. The class 
of entire functions for which $S(f)$ is a finite set  was  denoted
 $\classS$  by Eremenko and Lyubich \cite{MR1196102} in honor of Andreas Speiser. 
We let $\classS_n$ denote the set of 
entire functions with exactly $n$ singular values.
The Speiser class is a subclass of the Eremenko-Lyubich 
class  $\classB$, consisting of those entire functions 
whose singular values are a bounded set in $\complex$.
The two classes are also called ``finite-type'' and 
``bounded-type''  in holomorphic dynamics.

A natural measure of the growth of an entire function is 
its order: 
$$ 
\rho(f) = \limsup_{z \to \infty} \frac { \log_+ \log_+ |f(z)|}{ \log_+ |z|},$$
where $\log_+ r = \max (0, \log r)$. The natural parameter spaces 
of entire functions (at least for dynamical considerations) are the
quasiconformally equivalent functions: we say   $f,F$  are 
equivalent if there are quasiconformal maps $\phi, \psi$
of the plane    so that
$$   \psi  \circ f = F \circ \phi.$$
Eremenko and Lyubich \cite{MR1196102} proved that for $f \in \classS_n$, 
the collection of  functions equivalent  to $f$  forms a $n+2$ complex  dimensional 
manifold, and it is natural to ask if the order is constant
on each such manifold.  This is true for   $n=2$ 
 (e.g.  Proposition 2.2 of \cite{Epstein-Rempe} shows $\psi$, $\phi$ 
can be chosen affine in this case), but we will show:

\begin{thm} \label{Order conj}
There are equivalent  functions in $\classS_3$ with different orders.
\end{thm}

 More generally,    
we say that $f$ and $F$ are topologically
 equivalent if $\psi  \circ f = F \circ \phi$
for some pair of homeomorphisms $\psi, \phi$ of $\reals^2$ to itself.
The order conjecture asks if  $\rho(f) = \rho(F)$ whenever $f$ and 
$F$ are topologically equivalent. 
 For $f \in \classS$,  topological equivalence  is 
the same as quasiconformal equivalence
 (e.g., Proposition 2.2 of \cite{Epstein-Rempe}), 
but in general the two notions can differ. For meromorphic functions, the order 
can be defined using the Nevanlinna characteristic, and in this setting 
Kunzi  \cite{MR0069893} constructed a meromorphic function with finite singular 
set whose order can change under a topological equivalence.
Mori's theorem (e.g., Theorem III.C in \cite{MR2241787}) implies that 
a $K$-quasiconformal equivalence can change the order by 
at most a factor of $K$.  Also, functions in $\classB$ always 
have order $\geq 1/2$ (\cite{MR1344897}, \cite{MR1242092}, \cite{MR1853780}).

The order conjecture was formulated by Adam Epstein (e.g., \cite{Epstein-2007}, 
\cite{Eremenko-Nev})
 on the basis of several  examples 
 where  $\rho(f)$ can  be computed from combinatorial 
data associated to $f$.
For example, if $f''/f'$ is polynomial of degree $d$ then $\rho(f) = d$ 
 and the order  is 
the same for any entire function topologically equivalent to $f$.
Another example comes from dynamics: if
$p$ is a polynomial  of degree $d$ with a repelling fixed point at $0$ with 
multiplier $\lambda$ then there is an entire function $f$  (called 
the Poincar{\'e} function) so that $f(0)=0, f'(0)=1$ and 
$ f(\lambda z) = p(f(z))$. This function has order $\rho(f) =  1/\log_d |\lambda|$
and the singular set of $f$ is the closure of the critical orbits of $p$. 
Thus if $p$ is post-critically finite, $f \in \classS$.
 In \cite{Epstein-Rempe},
Epstein and  Rempe prove that the order conjecture holds 
for such $f$. 
 If the post-critical set of $p$  is bounded (which is the same as 
saying its Julia set is connected), then $f \in \classB$.
By taking two quasiconformally 
conjugate polynomials with repelling fixed points at $0$, but with 
multipliers of different absolute values, they show the order conjecture 
fails in $\classB$.

We say a function $f \in \classB$ has the area property if 
$$ \iint_{f^{-1}(K) \setminus \disk }  \frac 1{|z|^2} dxdy < \infty,$$ 
whenever $K$ is compact subset of $\complex \setminus S(f)$.
The area conjecture asks if every function in $\classB$ has 
this property.
This question was first raised by  
Eremenko and Lyubich  \cite{MR1196102} in the  special case 
when $f$ has no finite asymptotic values. 
 Epstein and Rempe  prove in \cite{Epstein-Rempe}  
that the Poincar{\'e} functions associated to polynomials with Siegel disks 
do not have the area property; this gives counterexamples in $\classB$.
If $f \in {\classS}$  has the area 
property then it also satisfies the order conjecture (e.g., Theorem 1.4 of 
\cite{Epstein-Rempe}).
Thus  Theorem  \ref{Order conj} gives a counterexample 
to the area conjecture in $\classS$.
Even stronger examples are given 
in \cite{Bishop-classS}, e.g.,   a function $f \in \classS$
with $S(f) = \{ -1,0,1\}$ and  
so that $ \{ z: |f(z) |> \epsilon \}$ has finite Lebesgue area for any 
$\epsilon> 0$. 

During a visit to Stony Brook in April of 2011,  Alex Eremenko asked me 
a question about the geometry of polynomials with only two critical 
values. 
This led to  further discussions of the classes $\classB$ and $\classS$
with him and Lasse Rempe, and the current paper is among the 
consequences.  I thank both of them for 
their lucid explanations of known results  and generously sharing  
their ideas about open problems. 
I am  grateful to David Drasin for his careful reading of an 
earlier draft of this paper. His comments  prompted 
a  re-write of several sections and  improved the exposition 
and mathematics throughout the paper.
Also thanks to the referee for mathematical  corrections and 
suggestions to improve the exposition.

We use the notation $\disk$ for the unit disk, 
$\complex$ for the complex numbers, $\reals$ for the real numbers
and $\rhp$ for the right half-plane. When two quantities $x,y$ 
depend on a common parameter, $x \lesssim y$ means that $x$ is 
bounded by a multiple of $y$, independent of the parameter. 
We sometimes use the equivalent notation $x = O(y)$. If 
$x \lesssim y$ and $y \lesssim x$, we write $x \simeq y$.
We let $|E|$ denote the diameter of a set $E$.

\section{The basic idea }   \label{the idea} 

We first describe how to construct entire functions with 
exactly two critical values. At the end of the section 
we modify this idea to give a function with three critical 
values, and this is the type of function we will use to 
disprove the order conjecture.

Let $U = \complex \setminus [-1,1]$ and recall that the  map 
$$ z \to  \cosh(z) = \frac {e^z + e^{-z}}{2}, $$
acts as a covering map from $\rhp$ to $U$, 
with the half-strip 
$$ S =\{ x+iy: x> 0, |y| < \pi \} $$
being a fundamental domain.  The role of  $\cosh(z)$ in 
this paper is analogous to the role of $\exp(z)$ 
in \cite{MR1196102} 
where Eremenko and Lyubich consider functions whose
critical values are contained in a disk and $\exp(z)$ provided 
the covering map for the complement of this disk.
We will also use the conformal map of $S$ to $\rhp$ given 
by 
$$ z \to i \pi \sinh(z/2).$$
This is symmetric with respect to $\reals$ 
and fixes the boundary points $- i \pi, 0, i \pi$.

The idea of this paper is to build an entire function  starting 
from a simply 
connected  subdomain $\Omega  \subset S$.
We will assume that 
  $\Omega$ is symmetric with respect to the real line
 and  that it is obtained from $S$ by removing 
       finite trees rooted along the top and bottom sides of $S$.
       The edges of the trees will be line segments.
       See Figure \ref{Idea1}.

The vertices of $\partial \Omega$ form a locally finite set that 
includes all  vertices of the removed trees (including the points 
where the trees are attached to the sides of $S$), but may include other 
points on the tree edges or on the sides of $S$. We assume the origin 
is a vertex.  Note that $\partial \Omega$ 
is an infinite tree and hence is bipartite, i.e., we can label the 
vertices with $+1$ and $-1$ so that  no two adjacent vertices having the 
same label. 
Let $V_+, V_-$ denote the vertices labeled $+1$ and $-1$
respectively.

 Since $\cosh$ is $1$-to-$1$ on $S$, 
it is a conformal map from $\Omega$ to $\Omega'=\cosh(\Omega)$.
Let $\tau: \rhp \to \Omega$   be a conformal map that 
is also symmetric and fixes $0$. Then 
\begin{eqnarray} \label{first try}
 f(z) = \cosh(\tau^{-1} (\cosh^{-1}(z))),
\end{eqnarray}
is  holomorphic  from $\Omega'$ to $U$. The region $\Omega'$ is 
dense in the plane, so if we knew that 
$f$ extended continuously to $\partial \Omega'$,   we could 
deduce it is  entire (since the boundary of $\Omega'$ is 
a union of smooth arcs, it is removable for 
continuous, holomorphic functions, e.g.,  
\cite{MR1315551}, \cite{MR1785402}). 

\begin{figure}[htb]
\centerline{
\includegraphics[height=1.25in]{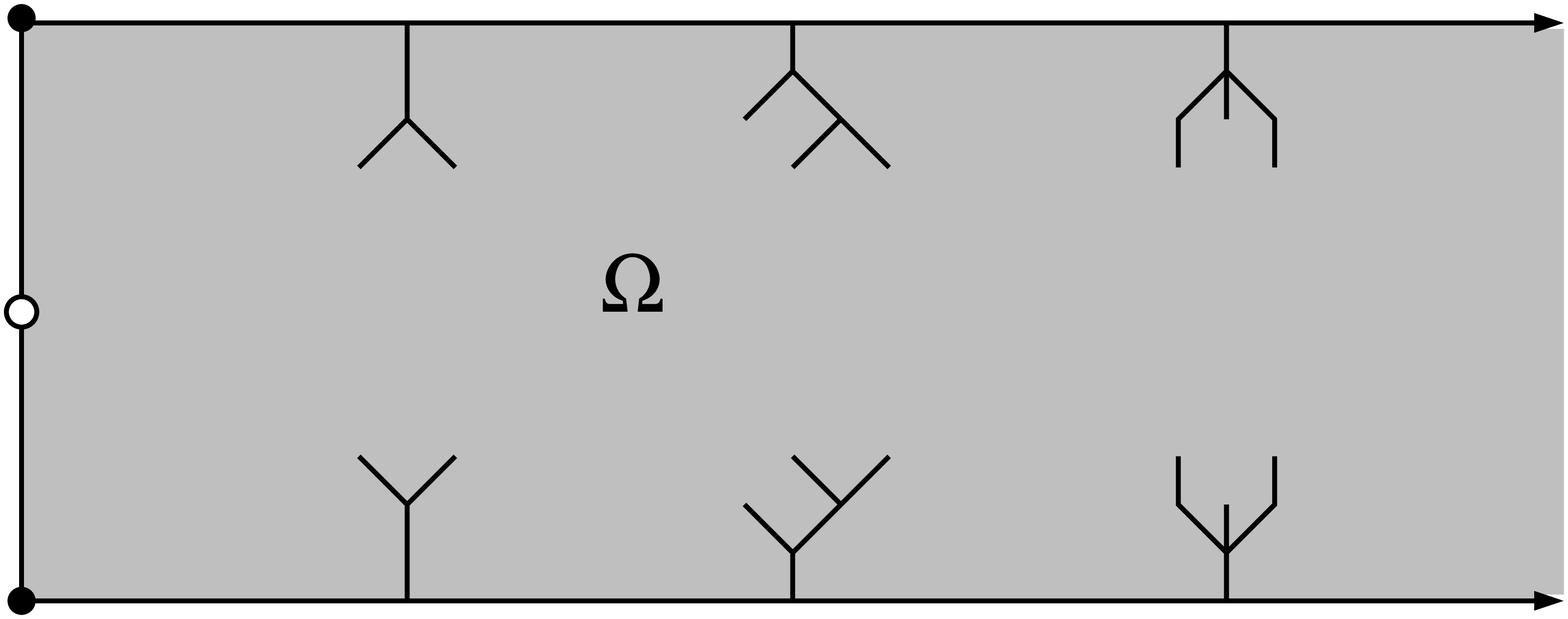}
$\hphantom{xxx}$
\includegraphics[height=1.25in]{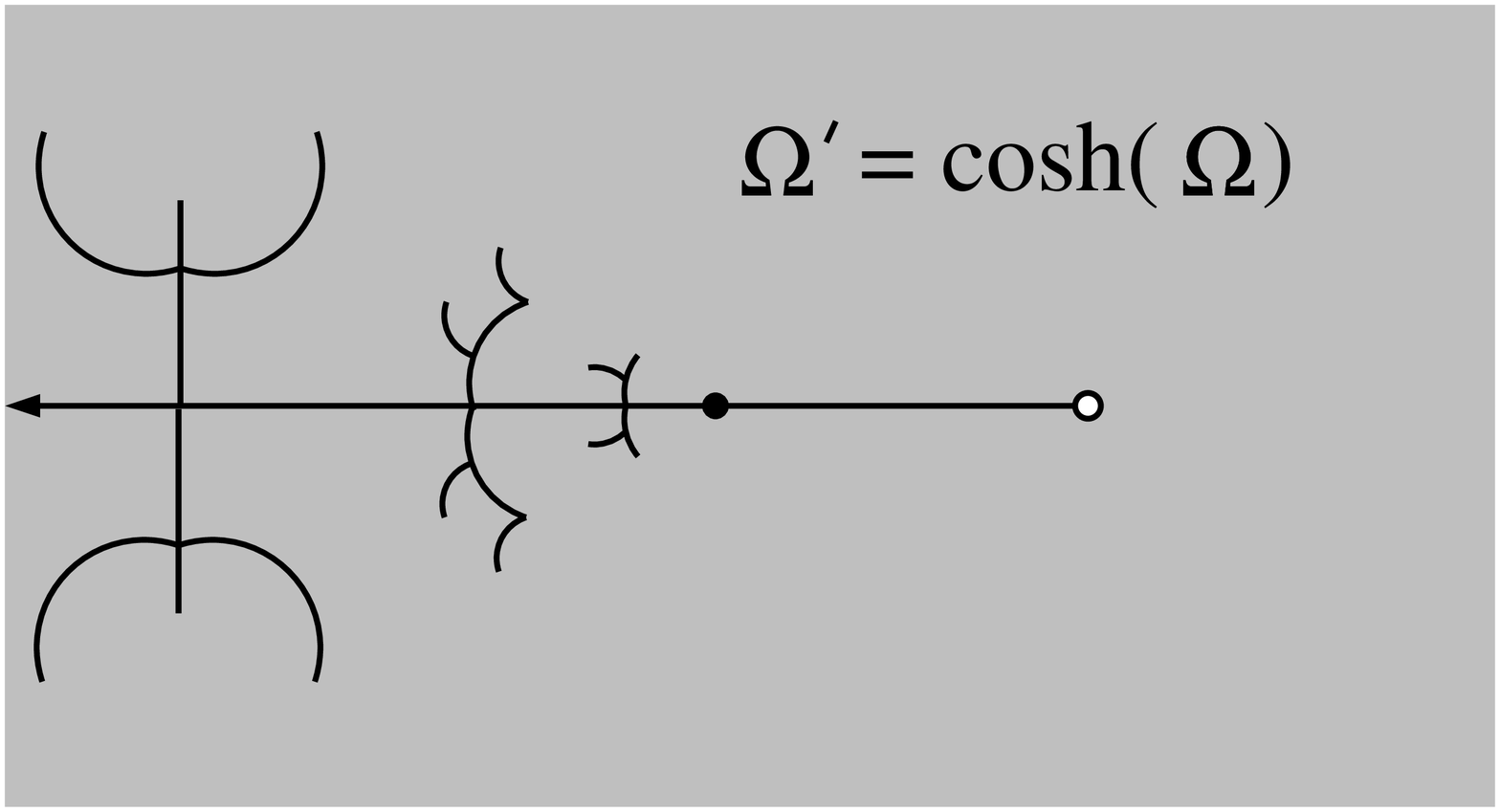}
}
\caption{ \label{Idea1}
 $\Omega$ is obtained by removing finite, straight edge trees 
from the half-strip $S$.  $\Omega' = \cosh(\Omega)$ is 
dense in the plane. Our function $f$  is holomorphic on $\Omega'$, 
but is it continuous across the boundary?
}
\end{figure}

However, $f$ is very unlikely to be continuous across 
$\partial \Omega'$. Suppose $z \in \partial \Omega$ is 
an interior point   of one of the trees we have
removed from $S$. Then  $\tau$ will map at least two different 
points of $\partial \rhp$ to $z$.
 This will cause 
a discontinuity for $f$ unless the final $\cosh$ in (\ref{first try})
maps all  $\tau$-preimages of $z$ 
 to the same point in $[-1,1]$.  In other words, for 
$f$ to extend continuously to the whole plane, we need 
\begin{eqnarray*} 
 \tau(ix) = \tau(iy)   &  \Rightarrow  & \cosh(ix) = \cosh(iy) \\
 &  \Leftrightarrow  & \cos(x)=\cos(y)  \\
 &  \Leftrightarrow  & \dist(x, 2 \pi \integers) =  \dist(y , 2 \pi \integers)
\end{eqnarray*}
 This will not happen in general, but our goal is to 
construct examples where is does happen 
by replacing the conformal map $\tau $
by a quasiconformal map $   \psi : \rhp \to \Omega$
with the essential property 
 \begin{eqnarray} \label{second try}
 \psi(ix) = \psi(iy)  \quad \Rightarrow \quad   \cos(x) = \cos(y).
\end{eqnarray}
When this holds we say  $\psi$ ``correctly identifies'' points.
In this case, the  function 
$$ g(z) = \cosh(\psi^{-1} (\cosh^{-1}(z)))$$
is quasiregular on $\Omega'$ and extends continuously across 
$\partial \Omega'$, and  hence is quasiregular on the whole plane. The 
measurable Riemann mapping theorem then implies there is a
quasiconformal map $\phi: \complex \to \complex$ so that 
$ f = g \circ \phi$
is entire.

Where are the critical values of $f$? The function $g$ is locally 
$1$-to-$1$ on $\Omega'$ and $\phi$ is $1$-to-$1$ everywhere, so
$f = g \circ \phi$ has no critical points on $\Omega'$. Thus 
all its critical points are in $\partial \Omega'$ and hence 
all critical values lie in $[-1,1] = f(\partial \Omega')$.
The only  critical values due to $\cosh$ are $\pm 1$; any others
must correspond to critical points of $\psi ^{-1}$ on $\partial \Omega$
and these can  occur only at vertices of $\partial \Omega$. Therefore
we also assume 
\begin{eqnarray} \label{white}
      \psi(ix) \in V_+  \Leftrightarrow x \in 2 \pi \integers, \qquad 
      \psi(ix) \in V_-  \Leftrightarrow x \in  \pi + 2 \pi \integers
\end{eqnarray} 
If both   and (\ref{second try}) and (\ref{white}) hold,
 then $f$ is entire and only has critical values $\pm 1$. 
Moreover, (\ref{second try}) can be reduced to a much easier 
condition to check.
Let ${\cal Z}$ be the partition of $i \reals = \partial \rhp$ into 
segments with endpoints  $\pi i \integers$. 

\begin{lemma} 
If (\ref{white}) holds and $\psi$ is 
linear on each segment  in ${\cal Z}$, then (\ref{second try}) holds.
\end{lemma} 

\begin{proof}
Suppose $v = \psi(ix) = \psi(iy)$ and $x \ne y$. If $v$ is a vertex then 
 by (\ref{white})  either  both $x$ and $y$ are in 
$2 \pi \integers$ or both are in $\pi + 2 \pi \integers$. In either case, 
(\ref{second try}) holds. If   
$v$ is in interior point of an edge $e = [ z,w] \subset \partial \Omega$, 
say $v = tz+(1-t)w$.  If the two  $\psi$-preimages of $e$ containing 
$ix$ and $iy$  are 
$[ia,ib]$ and $[ic,id]$ respectively,    with $\psi(ia)=\psi(ic) = z$, then 
$x = ta+(1-t)b$, $y=yc+(1-t)d$. Since $a$ and $c$ are either both in 
$2 \pi \integers$ or  both in $\pi +2 \pi \integers$, the distances of 
$x$ and $y$ to $2 \pi \integers$ are the same,
 which implies (\ref{second try}).
\end{proof}

Our counterexample to the order conjecture will have three critical 
values instead of two, but the basic idea is the same as above.
The only difference is that we will build the domain $\Omega$ 
 so that its boundary vertices  
can be  3-colored with the labels $\{-1,0,1\}$  and so that the map 
$$  z \to \cosh ( \psi^{-1}(z)),$$
is well defined at each vertex of $\partial \Omega$ and sends 
each vertex  to the value of its label (each edge of $\partial 
\Omega$ is mapped to segment of length $\pi/2$ on $\partial \rhp$). 
To obtain such a $\psi$, not any 3-coloring will do; there
are two conditions that must be met. First,  as 
we traverse the boundary $\partial \Omega$,
the labels must occur in the order 
$$\dots,1,0,-1,0,1,0,-1,0,1,0\dots,$$
which is the same order we encounter them when traversing the boundary 
of $U = \complex \setminus [-1,1]$.
Second, no leaf of the tree (a vertex of degree 1) can 
have label $0$. Together, these conditions are necessary 
and sufficient.
Then  as before, 
$$ g(z) = \cosh(\psi^{-1} (\cosh^{-1}(z)))$$
extends to be  quasiregular on the plane and the measurable 
Riemann  mapping theorem gives a quasiconformal $\phi$
so that   $f = g \circ \phi$
is entire with critical values $\{ -1,0,1 \}$. 
Note that the  critical points with critical 
value $0$ must correspond to vertices of $\partial \Omega$ 
with label $0$ and degree $\geq 3$.

\section{ Exponential partitions }   \label{aligning} 

We just described how our construction of an entire function $f$
depends on the construction of a domain $\Omega \subset S$ and 
 quasiconformal map $\psi: \rhp \to \Omega$ that 
correctly identifies points. The map $\psi$ will be written as 
a composition  $\psi = \psi_1 \circ \psi_0$ where $\psi_0 : \rhp \to S$
is quasiconformal and piecewise linear on the boundary, 
and $\psi_1$ will be  piecewise linear from $S$ to $\Omega$.
The map $\psi_1$ is the ``interesting'' part and contains 
the essential geometry; $\psi_0$  simply approximates the   conformal 
map from $\rhp$ to $S$. 
In this section we describe $\psi_0$.

Since $\psi_1$ will be constructed to be piecewise linear, 
there is a partition of $\partial S$ into  a collection 
of segments ${\cal I}$ so that $\psi_1$ is linear  on 
each $I \in {\cal I}$ (these will be the preimages under
$\psi_1$ of the edges of $\partial \Omega$). 
The elements of ${\cal I}$ are taken to 
be closed line segments that cover $\partial S$ and 
 are pairwise disjoint except for endpoints.
In addition, we will assume ${\cal I}$ satisfies 
\begin{enumerate}
\item ${\cal I}$ is symmetric with respect to the real line (i.e., the top and 
        bottom edges of $S$ are partitioned in exactly the same way).
\item There is a $ s > 0 $ and a constant $C < \infty$ so that 
     for all $I \in {\cal I}$,
$$   |I| \simeq \exp( -  s \dist(I, I_0)) .$$ 
\end{enumerate}
We will call  ${\cal I}$  an  exponential partition of $\partial S$ if 
both these conditions hold.

Assume the elements of ${\cal I}$ on the top edge of $S$ are denoted 
      $I_1, I_2, \dots$ in left to right order. The elements  
      on the bottom edge are similarly denoted $I_{-1}, I_{-2}, \dots$.

\begin{lemma}
If ${\cal I}$ is an exponential partition, then $|I_m| \simeq \frac 1m$ 
for all $m \geq 1$.
\end{lemma}

\begin{proof}
For $k \geq 1$, let $N_k >0$ be the index of a  partition element that contains
$k + i \pi$ and let $n_k = N_{k+1} - N_k$ (this is approximately the
number of elements that hit $[k+i \pi, k+1 + i \pi]$). By assumption, 
$ n_k \simeq \exp(  s k)$ and hence (since a geometric sum
is dominated by its last term) $N_k \simeq n_k \simeq \exp( s k)$.
So  for an integer  $N_k \leq m  \leq N_{k+1}$, 
$$ |I_m| \simeq \exp(-  s k) \simeq \frac 1{n_k} \simeq \frac 1{N_k}
 \simeq \frac 1m.$$ 
\end{proof} 

Recall that ${\cal Z}$ denotes the partition of 
 $\partial \rhp = i \reals$  with endpoints $i \pi \integers$.
For $k\geq 1$, $Z_k = [(k-1)\pi, k\pi] \in {\cal Z}$ and $Z_{-k}= -Z_k$
(there is no interval labeled $Z_0$).

\begin{lemma} \label{psi0}
Suppose ${\cal I}$ is an exponential partition. Then there is a quasiconformal 
map $\psi_0: \rhp \to S$ so that ${\cal I} = \psi_0({\cal Z})$  and 
$\psi$ is linear on each  segment in  ${\cal Z}$.
\end{lemma}

\begin{proof}
Partition $\rhp$ as follows. Let $W_0$ be the region bounded by the vertical 
segment $[-2 \pi i, 2 \pi i]$ and an arc of the circle $|z|= 2 \pi$. 
For $k =1, 2 , \dots$, let
 $$W_k = \rhp \cap \{z:  \pi k \leq |z| \leq   \pi (k+1) \},$$ 
(this is a half-annulus).
Partition $S$ into rectangles by letting $S_k$, $k=0,1,\dots$ 
 be the points in $S$ that project 
vertically onto $I_{k+1} \in {\cal I}$.
Then it is easy to check  that we can map $W_0$ to $S_0$ by  a map that
is linear on each interval 
$Z_{-2}, Z_{-1}, Z_1, Z_2 \subset \partial W_0$ and that 
is linear from arclength on the circular side of $W_0$ to the right-hand 
side of $S_0$. See Figure \ref{FirstMap}.

Each  half-annulus  $W_k$, $k \geq 1$ can be quasiconformally 
mapped to $S_k$ by a map that is linear on each component of $\partial W_k \cap
\partial \rhp$ and is linear in the argument on the two circular sides. 
It is important to note that the quasiconstant can be chosen independent 
of $k$, but this is easy to see since both $S_k$ and $W_k$ are generalized 
quadrilaterals with modulus approximately $k$.
\end{proof}

\begin{figure}[htb]
\centerline{
\includegraphics[height=1.5in]{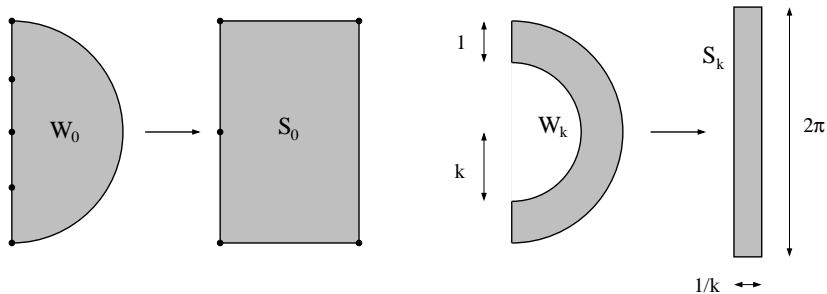}
}
\caption{ \label{FirstMap}
Define $\psi_0$  to be linear on boundary edges and uniformly quasiconformal.
}
\end{figure}

\section{A simple example}   \label{example} 
  
The  main  step in our construction is building 
the map $\psi_1: S \to \Omega$. 
This map will be built using  piecewise affine maps 
between combinatorially equivalent  infinite triangulations 
of $S$ and $\Omega$. Although the triangulations 
are infinite, only a finite number of different 
shapes are used (up to Euclidean similarity), and so 
the quasiconstant is bounded by the maximum over a finite 
set of affine maps.
In this section 
we will build such a map $\psi_1$ in a simple case, in order to 
introduce the basic idea before attacking the more 
complicated construction in the next section.

Triangulations of  two polygonal domains $\Omega_1, \Omega_2$ 
are compatible   if 
we have   a $1$-to-$1$ correspondence between  the  triangulations 
that  preserves interior  adjacencies  (i.e., if two triangles 
share in edge in $\Omega_1$ then the corresponding triangles share
an edge   in $\Omega_2$).
  See Figure \ref{compatible}.  We can then define a piecewise 
affine map from $\Omega_1$ to $\Omega_2$ by using the unique 
affine map between corresponding triangles that respects adjacency.
If triangles  share an edge on 
the boundary of $\Omega_2$ the corresponding triangles in 
$\Omega_1$ don't have to be adjacent. This will happen when   $\Omega_1$ 
has a Jordan boundary, but   $\Omega_2$ has slits.

\begin{figure}[htb]
\centerline{
\includegraphics[height=1.5in]{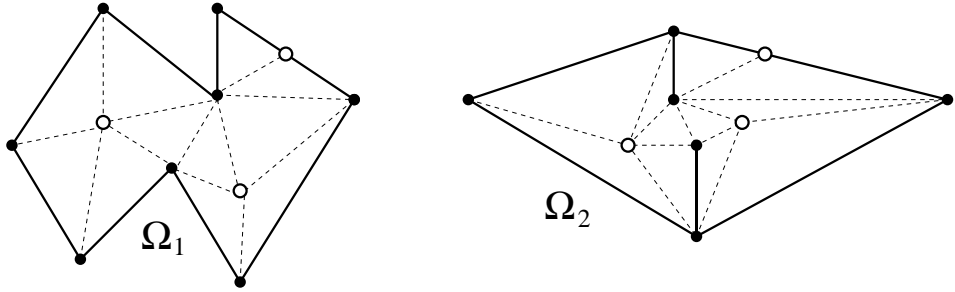}
}
\caption{ \label{compatible}
Compatible triangulations of two domains. Solid lines 
indicate the domain boundaries and dashed lines are  the 
interior edges of the triangulations.  
We allow extra vertices either 
in the interior or on the boundary.   Also note
that adjacency only need be preserved when triangles
share an edge inside the domain.
}
\end{figure}

\begin{figure}[htb]
\centerline{
\includegraphics[height=4.0in]{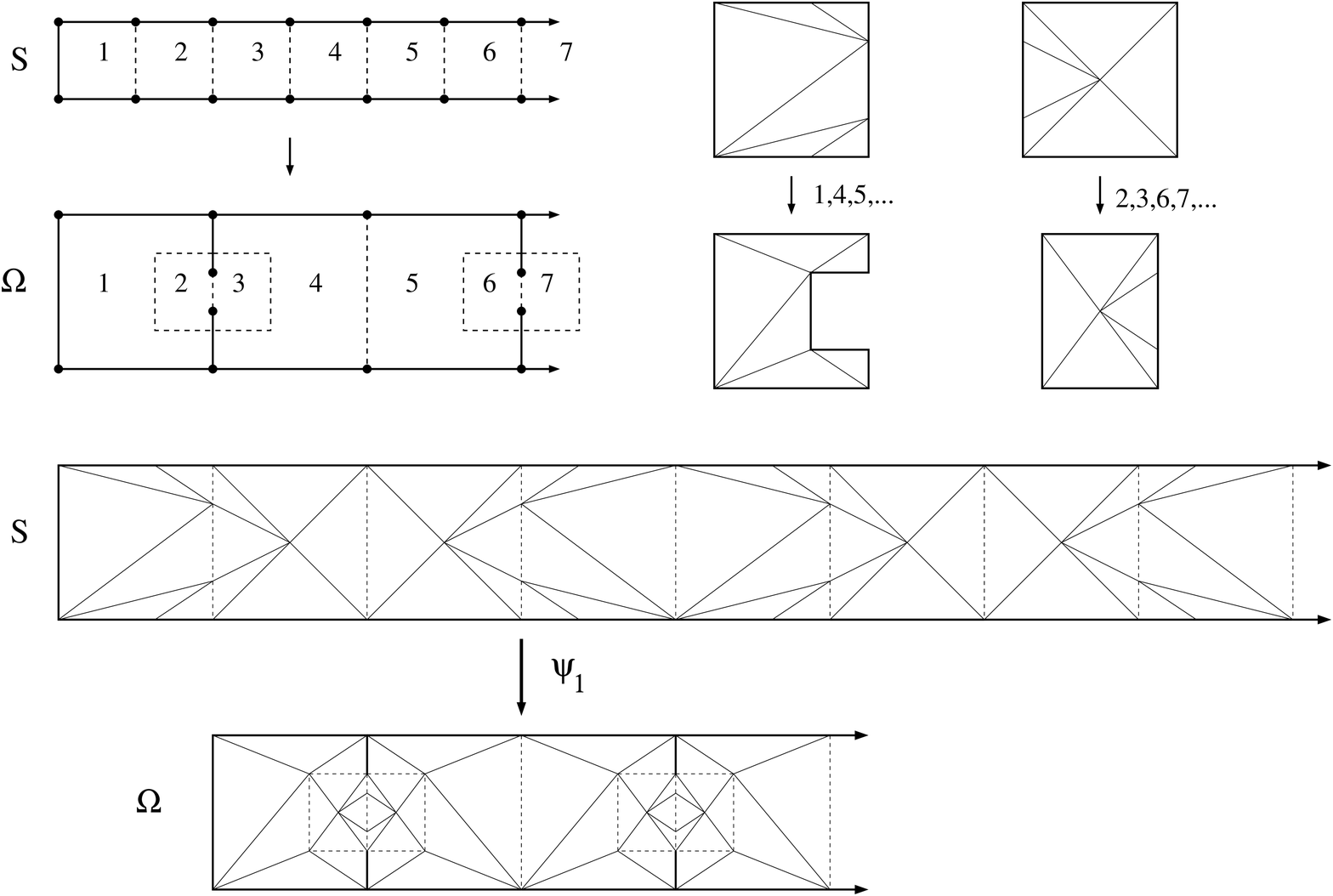}
}
\caption{ \label{FirstExample}
Constructing a piecewise affine map  $\psi_1:S \to  \Omega$   using 
compatible triangulations. On the upper left we show 
partitions of $S$ and $\Omega$ into corresponding subdomains and 
on the right we show that corresponding subdomains have
compatible triangulations. On the bottom we 
insert the triangulations into the pieces to show the 
compatible triangulations of $S$ and $\Omega$.
The same pattern is repeated forever.
}
\end{figure}
 
Now for the example.
 Consider Figure \ref{FirstExample}. It  shows the 
half-strip $S$ and a subdomain $\Omega \subset S$ formed by 
removing certain vertical slits from $S$.
 Both $S$ and $\Omega$ are 
partitioned into pieces that are numbered $1,2,3, \dots$ moving left to right.
 The pieces of $S$ are all squares.  The pieces of $\Omega$
have two shapes: rectangles and pieces that look like a letter ``C'' or 
its reflection.
Figure \ref{FirstExample} shows a triangulation for each piece 
and we can easily check that corresponding pieces for $S$ and 
$\Omega$ have compatible triangulations.
Using these building blocks we can find 
compatible triangulations of all of $S$ and $\Omega$ and hence a 
piecewise linear map from $S$ to $\Omega$. Since we are repeatedly 
using just two building blocks, the map is quasiconformal, with 
constant determined by the most distorted triangle from a 
finite set of possibilities. Note that the particular choice of 
triangulation is not important, and we have made no attempt to 
optimize the number of triangles used or the resulting QC constant.

The map $\psi_1$ is linear on a partition ${\cal J}$ 
 of $\partial S$ consisting of 
equal length segments.   We define an exponential partition 
by dividing the segments on $\partial \Omega$ into sub-segments;
about $2^n$ for segments that are between distance $n$ and $n+1$
from the left side of $\Omega$. Pulling back  these edges  
 on $\partial \Omega$  via $\psi_1$
gives an exponential partition of $\partial S$ and then we 
apply the construction of the Section \ref{aligning} to build $\psi_0$. 
If $ \psi = \psi_1 \circ \psi_0 : \rhp \to \Omega$ then 
$$ g= \cosh\circ \, \psi^{-1} \circ \cosh^{-1} $$
is defined on $\Omega' = \cosh(\Omega)$ and 
 extends to a quasiregular function on the plane. 
Thus  there is  a   $\phi$  so that
$ f = g \circ \phi$ is entire with two critical values.

Also note that $f$  has  positive, finite order. It has 
order $\geq 1/2$ simply because all functions in 
class $\classB$ do. To show $f$ has finite  order it 
suffices to show $g$ does; the quasiconformal correction 
can only change the order by a multiple depending on 
the quasiconformal constant of  the correction 
map $\phi$ (i.e., $\phi$ is H{\"o}lder). The main 
property of $\Omega$ that we need is that it contains 
a half-strip $\{ x+iy: x> 0, |y| < w\}$ for some fixed 
$w >0$. 

Suppose $x \geq 2 \pi$ and let $D=D(x,r)$ be the maximal
disk centered at $x$ contained in $\Omega$. Since $\Omega$
is contained in the strip $S$, we have $r \leq \pi$ and 
$\partial D$ hits $\partial \Omega$ at symmetric points 
$ t\pm i y$ with $|x-t| \leq \pi$  and $ y \geq w$. Let
$S_x$ be the hyperbolic geodesic in $D$ connecting these 
two points. This arc cuts $D$ into two subdomains that are
quasicircles with uniformly bounded constant, and hence 
a result of Fern{\'a}ndez, Heinonen and Martio
\cite{MR981499}  implies the image of each of
these domains is a quasicircle with uniformly bounded 
constant in $\rhp$. This easily implies that 
$\gamma_x = \psi^{-1}(S_x)$  is a curve in $\rhp $ that   is  
symmetric with respect to the real line and 
satisfies 
 $R_x \simeq r_x \simeq \psi^{-1}(x)$ where 
   $$  R_x = \max_{z \in \gamma_x} | z |, \qquad
  r_x = \min_{z \in \gamma_x} | z| .$$
(In other words, $\gamma$ is contained in an annulus around 
the origin of fixed modulus, independent of $x$.)
See Figure \ref{FHMlemma}.

\begin{figure}[htb]
\centerline{
\includegraphics[height=2.0in]{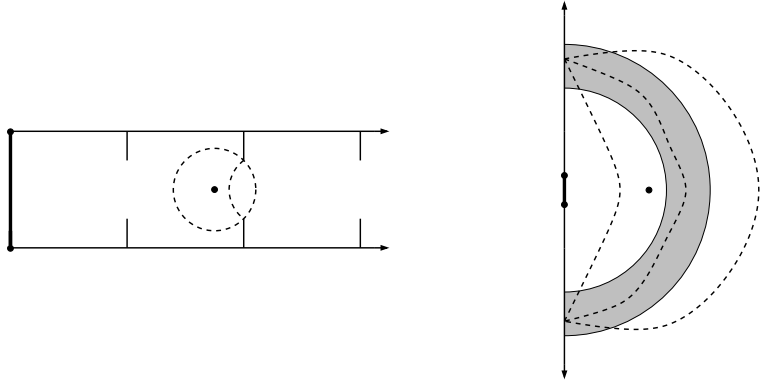}
}
\caption{ \label{FHMlemma}
For each $x$ we can find a nearby cross-cut 
of $\Omega$ that maps to an approximate
half-circle in $\rhp$, centered at the origin, and 
an extremal length argument shows 
the logarithm of the diameter grows
linearly with $x$ if $\Omega$ contains an 
infinite strip. This implies  $f$ has finite 
order. 
}
\end{figure}

If $I$ is the left side of $\Omega$ then 
 the extremal distance from $ \psi^{-1}(I)$ to $\gamma_x$
in $\rhp$ 
is comparable to the extremal distance from $I$ to $S_x$ 
in $\Omega$
(since modulus is quasi-preserved by $\psi$).
The first  is easily seen to be comparable to 
$\log R_x$ and the second is comparable to $x$, since 
$\Omega$ is contained in the half-strip $S$  and contains another 
half-strip  of positive width. 
Finally, 
   $$  \rho(g) \leq \limsup_{ x \to \infty}
   \frac {\max_{u+iv \in \Omega, u <x} \log |\psi^{-1}(u+iv)|}{x}
 \leq \limsup_{ x \to \infty} \frac {\log R_{x+\pi}}{x},$$
and our previous remarks show the rightmost term  is bounded as $x \to 
\infty$.

More generally, this argument shows that $g$ (and hence $f$)
will have finite order whenever the domain $\Omega$ contains 
an infinite half-strip of positive width.

\section{The  main construction }   \label{main}

The domain $\Omega$ used in the proof of Theorem \ref{Order conj} 
  is illustrated in Figure \ref{Domain}. It 
consists of the half-strip $S$ with finite trees removed along the 
integer points of the top and bottom edge. The domain is 
symmetric with respect to the real axis and in most of our 
pictures we will only show the lower half to simplify the 
illustrations.  Note that there is a  half-strip 
contained in $\Omega$; this implies that the 
function we eventually obtain will have positive, finite order 
by our previous comments.

\begin{figure}[htb]
\centerline{
\includegraphics[height=2.0in]{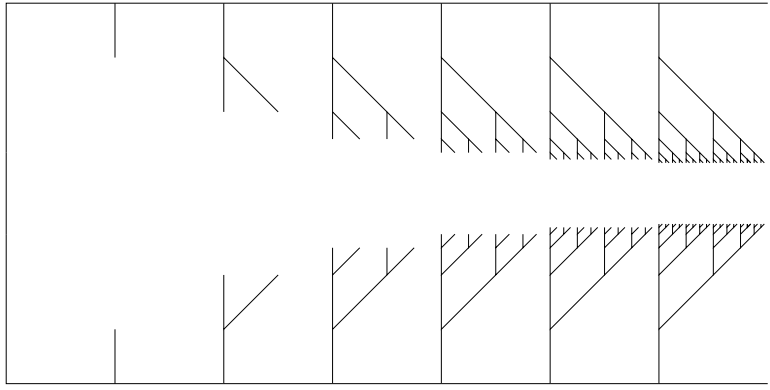}
}
\caption{ \label{Domain}
The domain in $\cosh$ coordinates. The domain is symmetric with 
respect to the real axis, and the trees are attached at the 
integer points. The line segments are further divided into 
edges by vertices of degree $2$; see Figure \ref{FourColors}.
}
\end{figure}

Every removed tree contains a vertical unit line segment attached 
to $\partial S$. For $n=1$ this segment is the entire tree.
For $n=2$ we add two segments: another vertical line segment 
of length $1$ and a diagonal segment of slope $1$ and
 length $\sqrt{2}$ (so the degree one 
vertices of the resulting tree are both distance $2$ from $\partial S$).
In general, the tree $T_n$  attached at the $n$th point is the vertical 
segment plus a binary tree of depth $n-1$ as shown in Figure \ref{Domain}.

The edges of  $T_n$ are naturally divided into levels from $0$
(attached to $\partial S$) to $n-1$ (adjacent to the leaves of the tree).
We form a new tree $T_n'$ by subdividing  edges in
 the $k$th level of $T_n$  into $4^k$ equal sub-edges. 
See Figure \ref{FourColors}. Note that  edges in the 
$k$-th level have length $ \simeq 2^{-k} \cdot 4^{-k} = 8^{-k}$.
 We define the vertices of $\partial \Omega$
to correspond to the vertices of $T_n'$. It is these new, shorter 
edges that will eventually be mapped to segments of length 
$\pi/2$ on $\partial \rhp$. We also subdivide the edges of 
$\partial \Omega$ on $\partial S$ by adding vertices where 
the real parts equal $n +\frac 12$, $n=1,2,3\dots$ along
the top and bottom edges of $T$ and adding the points $ \pm i \pi /2$
to the left side of $S$. 

 This subdivision of $T_n$ to obtain $T_n'$  is important for two
reasons. First, when we define the  map $\psi_1: S \to \Omega$, these
new edges will pull back to an exponential partition of $\partial S$, 
and this allows us to define $\psi_0$ as before. Second, since
every edge of $T_n$  of level $k$ is divided into $4^k$ sub-edges, 
we can  label the vertices of $ T_n'$ so that the vertices of $T_n$ 
all get label $ 1$, except for the root of $T_n$ (the
point where it is attached to $\partial S$) that gets 
label  $0$.
Moreover, the roots of the trees are  the only 
vertices labeled zero that are not degree 2 vertices. Hence 
these are the only preimages of $0$ that are critical points.
This fact 
 will be important when we construct 
a quasiconformal deformation that changes the order of our function.

\begin{figure}[htb]
\centerline{
\includegraphics[height=1.8in]{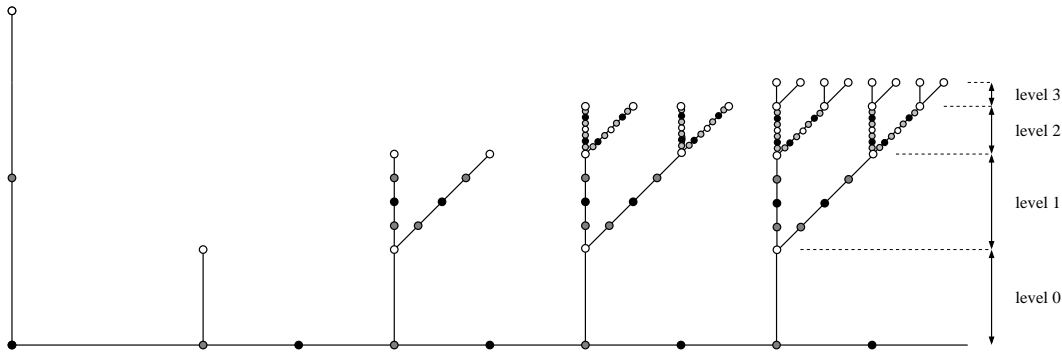}
}
\caption{ \label{FourColors}
$T_n'$ is  $T_n$ with extra vertices added.
After subdividing, the new vertices of    $\partial \Omega$ 
 can be three-colored as shown (black = -1, gray = 0, 
white = 1).
The labels are only shown for levels $0,1,2$ since edge lengths 
decrease exponentially with the level.
}
\end{figure}

Next we have to describe the map $\psi_1 : S \to \Omega$. Consider the 
dashed curve $\gamma$ in Figure \ref{Gamma2}. It is horizontal  along the 
top of each tree and has slope $1/2$ between trees (because 
the horizontal distance between $T_n$ and $T_{n+1}$ is 
$2^{-n+1}$ and $T_{n+1}$ is $2^{-n}$ ``taller'' than $T_n$.
 The curve $\gamma$
and its reflection across the real line bound a region $\Omega_1 
\subset \Omega$ (half of $\Omega_1$ is the shaded region in Figure 
\ref{Gamma2}).  It is easy to map $S$ to $\Omega_1$ by a piecewise 
linear quasiconformal map that sends $S\cap L$ affinely to 
$\Omega_1 \cap L$ for every vertical line $L$.

\begin{figure}[htb]
\centerline{
\includegraphics[height=1.5in]{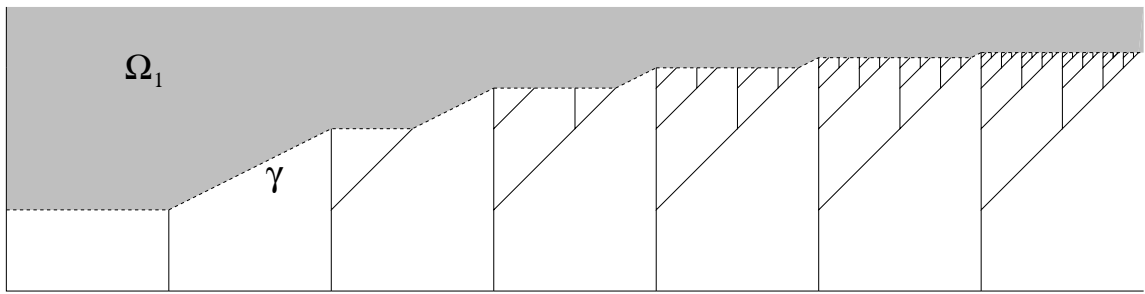}
}
\caption{ \label{Gamma2}
The first step is to map the half-strip $S$ to the variable width 
strip $\Omega_1$ by a piecewise linear quasiconformal map. This is 
easy.
}
\end{figure}

The leaves  of $\partial \Omega$ lie on $\partial \Omega_1$ and 
divide it into segments. We associate each such segment $I$ to a 
region  $Q_I \subset \Omega_1$.
See Figure \ref{Gamma3}.
 If $I$ is  horizontal, then $Q_I$ is the 
square in $\Omega_1$ with $I$ as one side. If $I$ has slope $1/2$
then $Q_I$ is the right triangle in $\Omega_1$ with hypotenuse $I$ and 
vertical and horizontal legs. 
We also associate to $I $ the component $U_I$ of $ \Omega \setminus \Omega_1$
that has $I$ on its boundary. 

\begin{figure}[htb]
\centerline{
\includegraphics[height=1.5in]{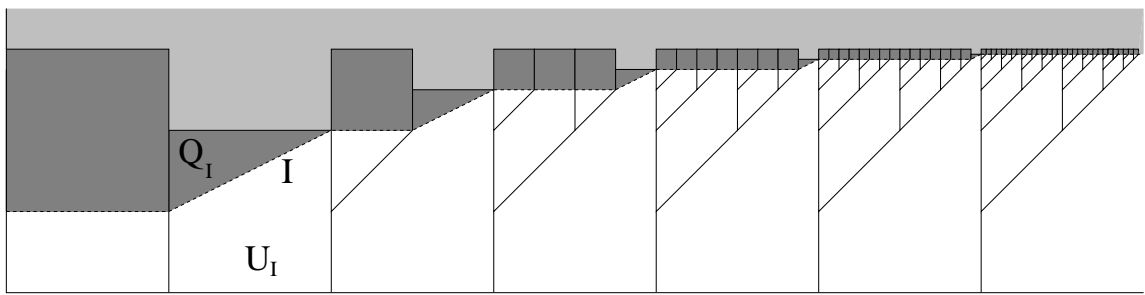}
}
\caption{ \label{Gamma3}
The leaves of $\partial \Omega$ partition $\partial \Omega_1$ and we associate 
to each segment  $I$ a region $Q_I$   above it  and  a 
region $U_I$ below it.
}
\end{figure}
\begin{figure}[htb]
\centerline{
\includegraphics[height=1.5in]{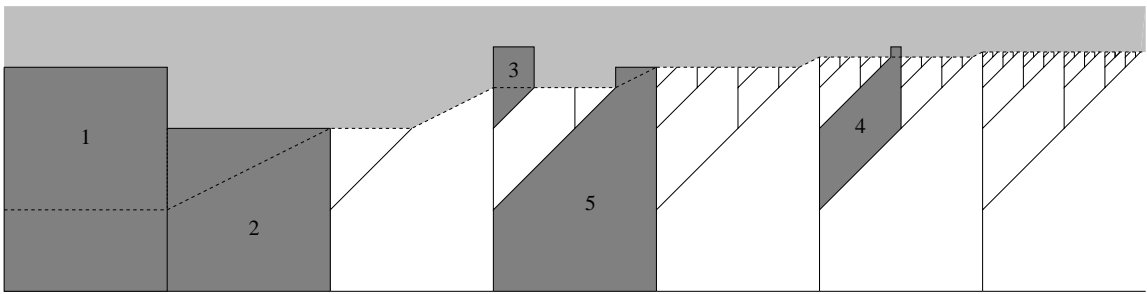}
}
\caption{ \label{Gamma2B}
The regions $Q_I$ are mapped to regions $W_i$ that ``fill in'' 
$\Omega \setminus \Omega_1$. There are two special regions that 
are used just once each and three that are used repeatedly.
The five types of regions are illustrated. Regions 1 and 2 are 
special cases that occur only once each and cases 3, 4 and 5 are 
called triangular, full and partial tubes respectively.
There is only one type of triangular region, but the 
full and partial tubes come in infinitely many 
versions. However each of these versions is built from 
a finite number shapes.
}
\end{figure}

  We let $W_I = Q_I \cup U_I \cup I$
 be the union of the regions above and below $I$.  
We will define a map $\Omega_1 \to \Omega$ by mapping each region $Q_I$ to 
$ W_I$. Our map will be the identity on $\partial Q_I \cap \Omega_1$, 
and hence extends continuously to all of $\Omega_1$ by setting it to the 
identity outside all the $Q_I$'s. This map is called a ``filling'' map for 
$W_I$.

There are five cases to consider: two special cases that occur once each, and 
three cases that occur infinitely often. The two special cases are the 
first two segments on $\gamma$; the ones that project vertically to $[0,1]$
and $[1,2]$. Figures \ref{Map1} and \ref{Map2}  show how to define these maps 
using triangulations.

\begin{figure}[htb]
\centerline{
\includegraphics[height=1.5in]{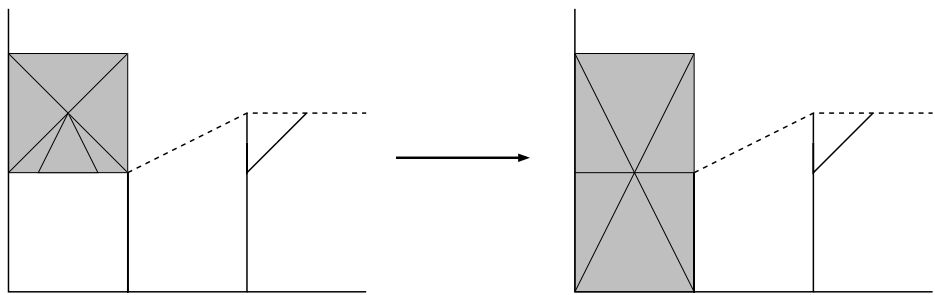}
}
\caption{ \label{Map1}
The first special case. 
}
\end{figure}

\begin{figure}[htb]
\centerline{
\includegraphics[height=1.5in]{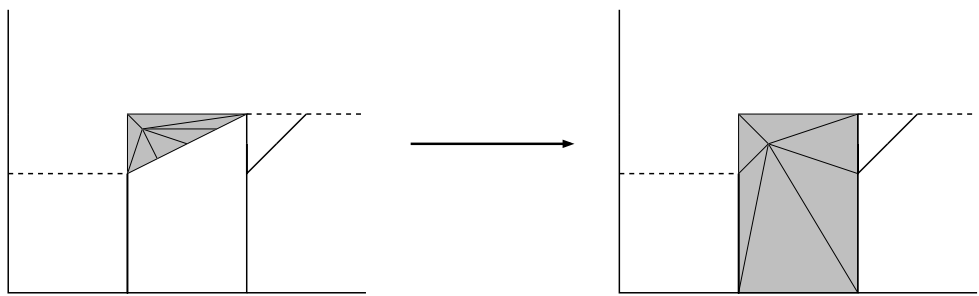}
}
\caption{ \label{Map2}
The second special case.
}
\end{figure}

The remaining three cases are illustrated 
 in Figures \ref{Map3}-\ref{Map8}. We refer 
to these cases as triangles ($U_I$ is a triangle), 
partial tubes ($U_I$ is not a triangle and 
 does not touch $\partial S$) and 
full tubes ($U_I$ touches $\partial S$).  
See cases 3, 4 and 5 in Figure \ref{Gamma2B}.

\begin{figure}[htb]
\centerline{
\includegraphics[height=2.0in]{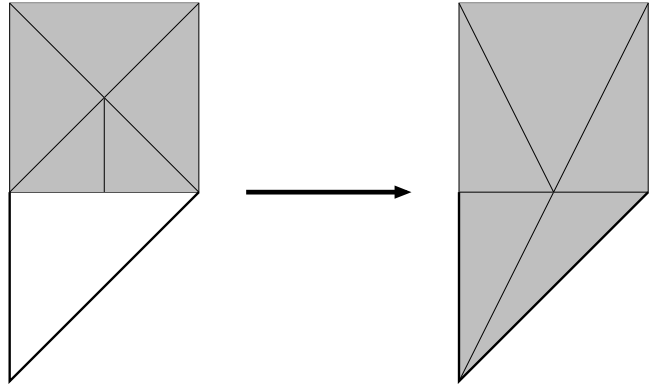}
}
\caption{ \label{Map3}
This shows the filling map for a triangular component. 
}
\end{figure}

The filling map for full tubes is illustrated in Figures
\ref{Map5A} and \ref{Map5B}. In Figure \ref{Map5A}, 
the left side shows the region $Q_I$ divided into 
a number of subregions; a pentagon on top, a series 
of non-convex hexagons and a final triangle.  The exact 
sizes of these regions is important and will be discussed
in the next paragraph; for the moment, just assume that 
each of the middle regions (the non-convex hexagons) 
is similar to all the others. The regions are labeled 
with numbers indicating their levels; these are the 
same levels we used to subdivide the edges of the tree 
$T_n$ to obtain the tree $T_n'$.

 The 
right side of Figure  \ref{Map5A} shows the region $W_I$  divided
into  subregions corresponding to the subregions of $Q_I$: a rectangle on top 
(this contains $Q_I$), a series of trapezoids, and a 
square on the bottom. The regions on left and right are 
numbered and each subregion on the left is mapped to the 
corresponding  subregion on the right  by a piecewise 
affine map. The triangulations that define this map 
are illustrated in Figure \ref{Map5B}.
Because we are only using three  basic shapes
(even though the middle shape may be used many times), 
each with a finite triangulation, the maps we build 
are clearly uniformly quasiconformal. 

\begin{figure}[htb]
\centerline{
\includegraphics[height=2.5in]{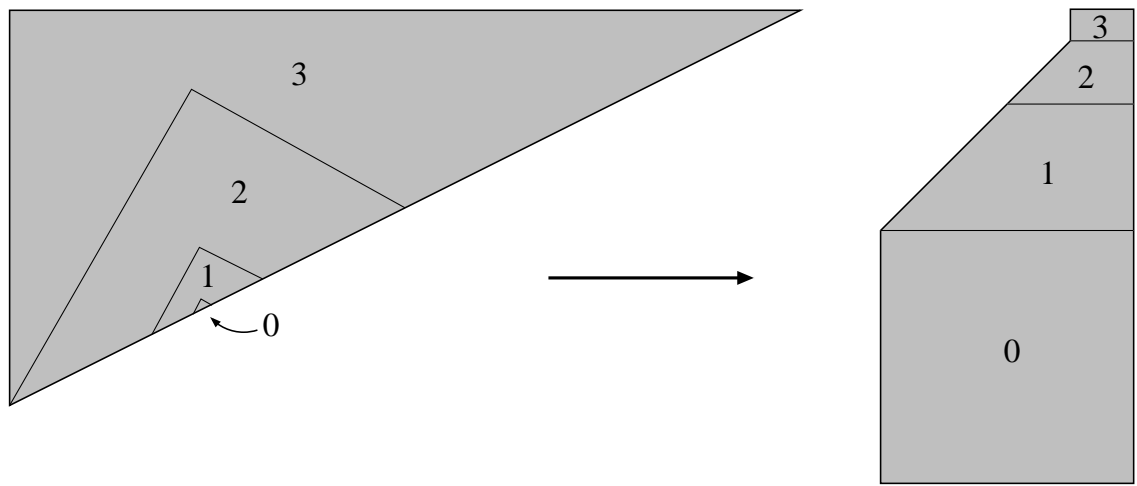}
}
\caption{ \label{Map5A}
This shows the filling maps for the full tubes. We subdivide $Q_I$ and 
$W_I$ as shown and map the pieces as illustrated in 
Figure \ref{Map5B}.
All the filling maps for full tubes are of this form, differing only in the 
number of times the central trapezoid piece is repeated.
The only critical feature is that components pieces in $Q_I$ have diameters 
that decay like powers of $4$ (this gives an exponential 
partition of $\partial S$).
}
\end{figure}

\begin{figure}[htb]
\centerline{
\includegraphics[height=3.0in]{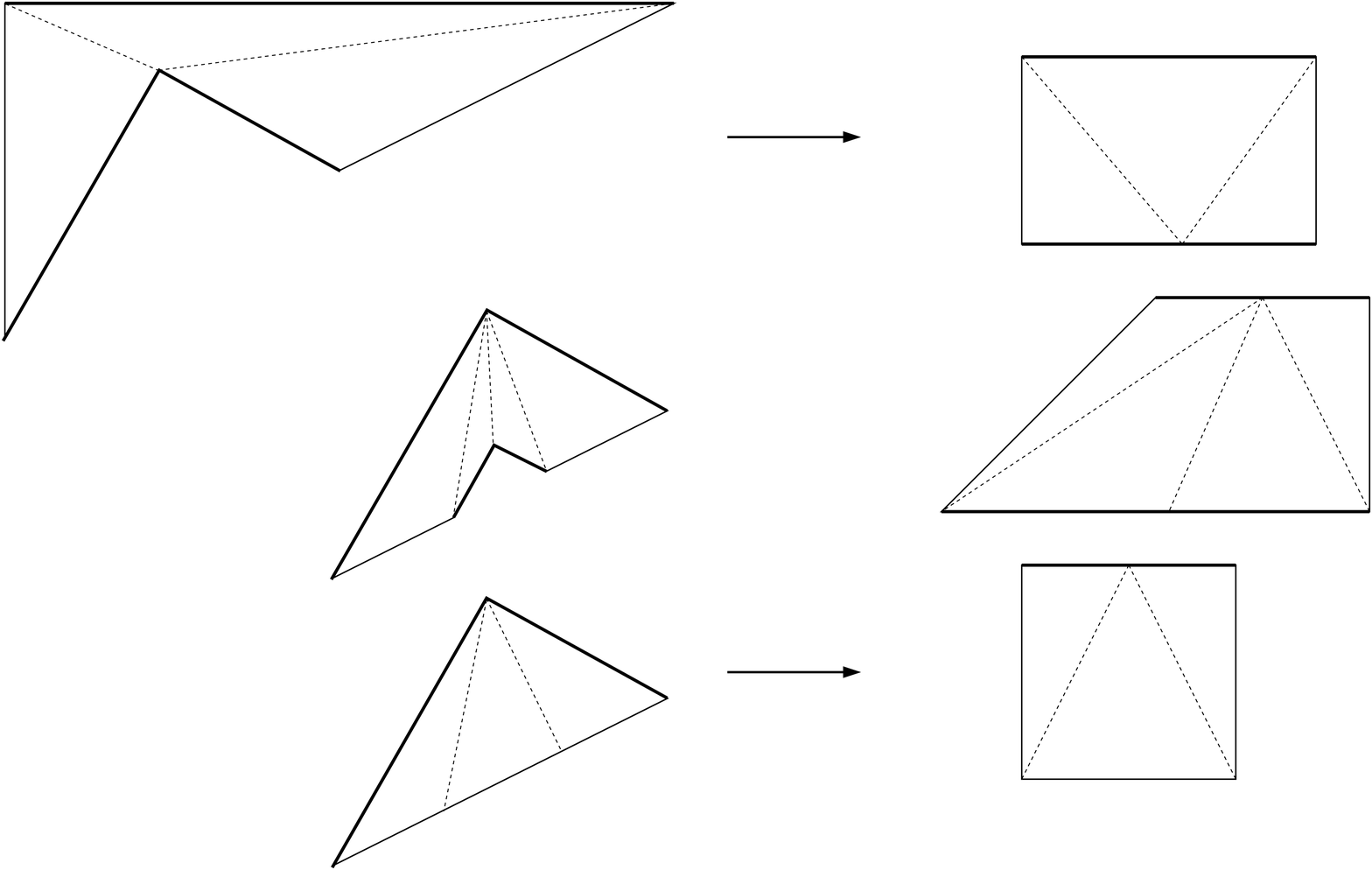}
}
\caption{ \label{Map5B}
This shows the triangulations defining the  filling maps for the full tubes. 
The reader should check that they are compatible.
}
\end{figure}

There is another property we need to check. The partition 
of the tree into edges is supposed to pull back to 
an exponential partition of the strip. This means that 
the vertices along the boundary of a full tube should pull 
back to a points on $I$ that are approximately evenly 
spaced and grow exponentially with the distance of 
$I$ from the left side of $\Omega$.  Consider Figure 
\ref{Map5A}. The full tube is pictured on the right and 
divided into levels (starting with the square on the 
bottom, which is level 0). The sides of level $k$ are 
divided into $4^k$ equal length edges of the 
tree (see Figure\ref{FourColors}).
These points pull back to $4^k$ equally  spaced points 
on  a subsegment of $I$ (on the left of Figure \ref{Map5A}, 
the segment is where the boundary of a subregion hits $I$).
In order for the collection of all preimage points to 
be equally spaced on $I$, we need each segment corresponding 
to a level $k$ region to  be 
$4$ times longer than the  segments corresponding to 
a level $k-1$ region, but this is easy to accomplish.
This implies that when the edges of $\partial \Omega$ are pulled 
back, every preimage in $[n,n+1]$  has length $\simeq 4^{-n}$  and 
hence they form an exponential partition.

The corresponding pictures for partial tubes are a little 
simpler. Figure \ref{Maps7} shows how we cut $Q_i$ and $W_I$ 
into corresponding pieces and Figure \ref{Map8} shows how 
each piece is mapped via a triangulation. The rest of the 
argument is the same as for the full tubes (checking the map 
is uniformly quasiconformal and defines an exponential partition 
of the strip).

This completes our construction of the quasiregular map 
$g=\cosh \circ \, \psi^{-1} \circ \cosh^{-1}$ and hence of 
an entire function $ f = g \circ \phi$ with three singular 
values. The next section 
will prove $f$ is a counterexample to the order conjecture.

\begin{figure}[htb]
\centerline{
\includegraphics[height=2.0in]{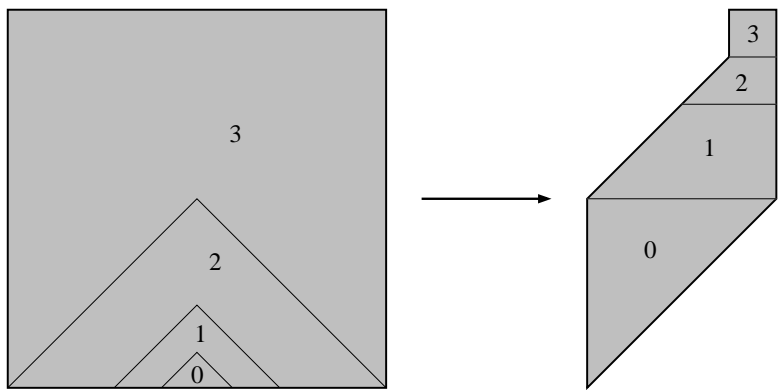}
}
\caption{ \label{Maps7}
This shows how to decompose $Q_I$ and $W_I$ into corresponding 
pieces to construct the filling maps for partial tubes. 
}
\end{figure}

\begin{figure}[htb]
\centerline{
\includegraphics[height=3.0in]{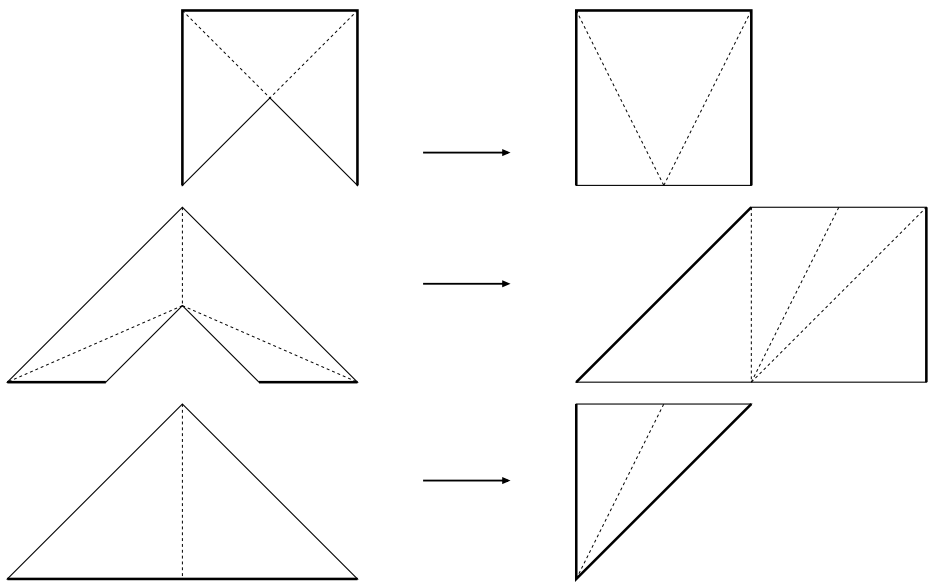}
}
\caption{ \label{Map8}
This shows how to map corresponding pieces of $Q_I$ and $W_I$ for 
partial tubes. 
}
\end{figure}

\clearpage

\section{The order conjecture fails }   \label{order conjecture} 

Suppose $f$ is the entire function constructed in the previous 
section. 
We claim   we can choose  quasiconformal maps $\tau, \sigma$ 
and an entire function $F$ so that 
$$  F \circ \sigma = \tau \circ f, $$
where $\tau$ is the identity off a  compact set and $\sigma$ satisfies 
\begin{eqnarray} \label{sqrt root}
 |\sigma(z)|  = O(\sqrt{|z|}),
\end{eqnarray}
for all sufficiently large $z$. 

First we check that  (\ref{sqrt root}) gives the desired
counterexample.  Taking $z = \sigma (w)$,  
\begin{eqnarray*}
 \lim_{  |z| \to \infty}  \frac {\log_+ \log_+ |F(z)|}{ \log_+ |z|}
  &=& 
 \limsup_{|w| \to \infty} \frac {\log_+ \log_+ |F(\sigma(w) )|}
           { \log_+ |\sigma(w)| } \\
  & \geq & 
 \limsup_{|w| \to \infty} \frac {\log_+ \log_+ | \tau(f(w) )|}
           { \log_+ | \sqrt{w}| } \\
  & \geq & 
 2 \limsup_{|w| \to \infty} \frac {\log_+ \log_+ | f(w )|}
           { \log_+ | w| }, \\
\end{eqnarray*} 
since $f = \tau \circ f$ when $|f|$ is large enough.
Thus the order of $F$ is at least twice the order of $f$. 
Since the order of $f$ is positive and finite, the two 
orders are different.

Now we prove the claims.   First we define $\tau$.  Let 
$W =  D(1,2)  \setminus  [0,1]$. 
$W$ is a topological annulus and 
is conformally equivalent to the round annulus    $A_t = 
\disk \setminus t \disk$  for some $t$.
For $K > 1$, the map 
$$  r_K: z \to z|z|^{K-1},$$
is a $K$-quasiconformal from $A_t$ to  $A_{t^K}$  that is  the identity on
 the outer boundary of $A_t$ and 
decreases the extremal length of the path family that 
separates the two boundary components of $A$ by a factor 
of  $K$.
Let $W_K$ be the domain of the form $D(1,2) \setminus [s,1]$ that 
is conformally equivalent to $A^{t^K}$.  Thus $r_K$ transfers to a 
$K$-quasiconformal map $\tau:W\to W_K$ that is the identity on the 
outer boundary of $W$
decreases the extremal length of the path family that 
separates the two boundary components of $A$ by a factor 
of  $K$.
 It is easy to check that $\tau$ extends 
continuously across the slit boundary of $W$ and can be extended 
by the identity   outside $\overline{W}$ to give a $K$-quasiconformal
map of the whole plane. 
 This is the map $\tau$ we use in 
our quasiconformal equivalence. By the measurable Riemann mapping 
theorem, there is a $K$-quasiconformal map $\sigma$ of the plane 
so that $F = \tau \circ f \circ \sigma$ is entire. 
We can assume $\sigma$ fixes $-1,1,\infty$.

All that remains is to show  that (\ref{sqrt root}) holds.
Let $V = D(1,2)\setminus [-1,1]$ and consider inverse images
of $V$ under $f$.  By construction, every such preimage contains 
two preimages of $0$ on its boundary and either 
\newline $ \hphantom{xxxx}$ (I)  both preimages are critical points,
\newline $ \hphantom{xxxx}$ (II) exactly one preimage is a critical point, or
\newline $ \hphantom{xxxx}$ (III) neither preimage is a critical point.
\newline 
Case I occurs when the preimages of zero are the roots 
of two adjacent trees. Case II only occurs twice and corresponds 
to the corners of $S$ in $\cosh$ coordinates. Case III 
occurs for preimages that do not touch $\partial S$ in $\cosh$
coordinates.
A cartoon of the different types of preimages is shown (in 
$\cosh$ coordinates) in Figure \ref{Dilatation}. 

\begin{figure}[htb]
\centerline{
\includegraphics[height=1.8in]{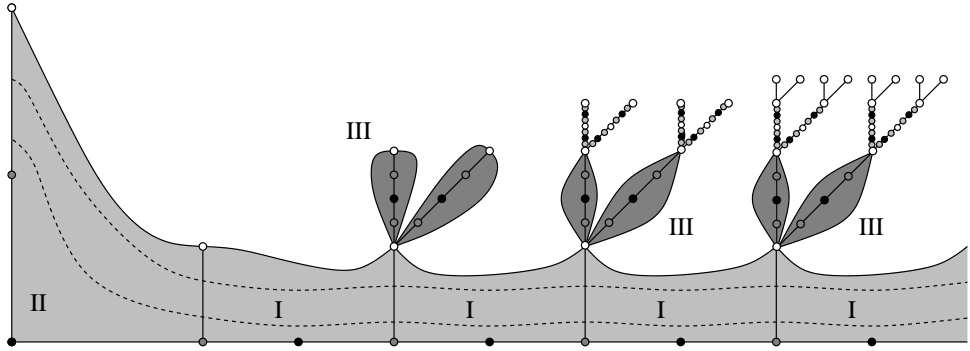}
}
\caption{ \label{Dilatation}
The support of our dilation lifted to $\Omega$ (its easier to 
see in $\Omega$ than in $\Omega'=\cosh(\Omega)$). The light gray
is the union of type I  and type II  preimages  and the darker gray shows a 
few  type III preimages (this is a sketch, not a computation).  The dashed
lines represent the  path family  $\Gamma_n$ 
whose extremal length is  decreased
by a factor of $K$ by our dilatation.
The grey dots correspond to preimages of $0$.
}
\end{figure}

Let $\gamma_0= [- i\pi /2, i \pi /2]$ and for $n=1,2,\dots$ let 
$\gamma_n$ be the vertical crosscut of $S$ 
that contains the point  where the $n$th tree is attached.
Let $\Gamma_n$ be the path family connecting $\gamma_0$  
 to $\gamma_n$ inside the closure of the  type I and type II components.
This family has extremal length $\leq Mn$ for some fixed $M$.

 Applying $\cosh$, $\gamma_n$, $n\geq 1$,  maps to an ellipse $E_n$ of 
bounded eccentricity and diameter $\simeq e^n$.  
Also, $\cosh(\Gamma_n)$ is a family of paths
connecting $[0,1]$ to $E_n$. Let $V_n$ be the region bounded
by the ellipse $E_n$,  minus $[-1,1]$ and let $\Sigma_n$ be the path 
family in $V_n$ connecting $[-1,1]$ to $E_n$. Then 
$\Sigma_n$ contains $\cosh(\Gamma_n)$, and $V_n \setminus (-\infty, -1]$
is conformally equivalent to a $2 \pi \times n$ rectangle, so 
if $\lambda $ denotes extremal length we get 
$$   n \simeq \lambda(\Sigma_n) \leq \lambda(\cosh(\Gamma_n))  
= \lambda(\Gamma_n) \simeq n.$$
When we  apply the quasiconformal map $\sigma$,
the extremal length of $\Gamma_n$ is reduced by a factor 
of $K$; this is exactly why we choose $\tau$  as we did. Hence 
 \begin{eqnarray} \label{EL upper}
\lambda (\sigma(\Sigma_n)) \leq \lambda(\sigma(\cosh(\Gamma_n)))  
\leq \frac {Mn}{K}, 
\end{eqnarray}
if  $K$ is large enough.
The eccentricity of the ellipse $E_n$  tends to $1$ 
as $n \nearrow \infty$ so $\sigma(E_n)$ is a $2K$-quasicircle when 
$n$ is large enough.  Let 
$$
R= \max_{\sigma(E_n)} |z|  , \qquad r =  \min_{\sigma(E_n)}    |z|. 
$$
Since $\{ 1\leq |z| \leq r\} \subset \sigma(V_n)$,  we have 
$ \frac 1{2\pi}  \log   r   \leq  \lambda (\sigma(\Sigma_n))
 \leq Mn/K $, 
and hence 
$$ r \leq \exp( 2 \pi M n /K) \leq   \exp(n/2) $$
if $K \geq 4 \pi M$. Quasiconformal maps are 
quasisymmetric, so we have $R \leq C_K \cdot r$ for some constant 
depending only on $K$ and hence $ R   \leq C_K \cdot  e^{n/2} . $
If  $z \in V_n \setminus V_{n-1}$, then 
$$ |\sigma(z)| \leq R= O(e^{n/2}) = O(e^{(n-1)/2}) = O(\sqrt{|z|}),
$$
which  is (\ref{sqrt root}).  This  proves 
Theorem \ref{Order conj}, i.e., 
the order conjecture fails in $\classS$.

\bibliography{order}

\def\cprime{$'$} \def\cprime{$'$}
\begin{thebibliography}{10}

\bibitem{MR2241787}
Lars~V. Ahlfors.
\newblock {\em Lectures on quasiconformal mappings}, volume~38 of {\em
  University Lecture Series}.
\newblock American Mathematical Society, Providence, RI, second edition, 2006.
\newblock With supplemental chapters by C. J. Earle, I. Kra, M. Shishikura and
  J. H. Hubbard.

\bibitem{MR1344897}
Walter Bergweiler and Alexandre Eremenko.
\newblock On the singularities of the inverse to a meromorphic function of
  finite order.
\newblock {\em Rev. Mat. Iberoamericana}, 11(2):355--373, 1995.

\bibitem{Bishop-classS}
Christopher~J. Bishop.
\newblock Constructing entire functions by quasiconformal folding.
\newblock preprint 2011.

\bibitem{Epstein-2007}
Adam~L. Epstein.
\newblock Finite order entire functions and meromorphic quadratic
  differentials.
\newblock preprint 2007.

\bibitem{Epstein-Rempe}
Adam~L. Epstein and Lasse Rempe.
\newblock On the invariance of order for finite-type entire functions.
\newblock preprint 2012.

\bibitem{MR1196102}
A.~E. Eremenko and M.~Yu. Lyubich.
\newblock Dynamical properties of some classes of entire functions.
\newblock {\em Ann. Inst. Fourier (Grenoble)}, 42(4):989--1020, 1992.

\bibitem{Eremenko-Nev}
Alex Eremenko.
\newblock Geometric theory of meromorphic functions.
\newblock preprint 2006.

\bibitem{MR981499}
Jos{\'e}~L. Fern{\'a}ndez, Juha Heinonen, and Olli Martio.
\newblock Quasilines and conformal mappings.
\newblock {\em J. Analyse Math.}, 52:117--132, 1989.

\bibitem{MR1315551}
Peter~W. Jones.
\newblock On removable sets for {S}obolev spaces in the plane.
\newblock In {\em Essays on {F}ourier analysis in honor of {E}lias {M}. {S}tein
  ({P}rinceton, {NJ}, 1991)}, volume~42 of {\em Princeton Math. Ser.}, pages
  250--267. Princeton Univ. Press, Princeton, NJ, 1995.

\bibitem{MR1785402}
Peter~W. Jones and Stanislav~K. Smirnov.
\newblock Removability theorems for {S}obolev functions and quasiconformal
  maps.
\newblock {\em Ark. Mat.}, 38(2):263--279, 2000.

\bibitem{MR0069893}
Hans~P. K{\"u}nzi.
\newblock Konstruktion {R}iemannscher {F}l\"achen mit vorgegebener {O}rdnung
  der erzeugenden {F}unktionen.
\newblock {\em Math. Ann.}, 128:471--474, 1955.

\bibitem{MR1242092}
J.~K. Langley.
\newblock On the multiple points of certain meromorphic functions.
\newblock {\em Proc. Amer. Math. Soc.}, 123(6):1787--1795, 1995.

\bibitem{MR1853780}
Gwyneth~M. Stallard.
\newblock Dimensions of {J}ulia sets of hyperbolic meromorphic functions.
\newblock {\em Bull. London Math. Soc.}, 33(6):689--694, 2001.

\end{thebibliography}
\bibliographystyle{plain}

\end{document}